\documentclass[12pt]{article}
\usepackage{lmodern}
\usepackage{microtype}

\usepackage{amsthm,amsmath,amssymb}
\usepackage[latin1]{inputenc}
\newtheorem{theorem}{Theorem}
\newtheorem{definition}[theorem]{Definition}
\newtheorem{lemma}[theorem]{Lemma}
\newtheorem{proposition}[theorem]{Proposition}

\newtheorem{example}[theorem]{Example}

\newtheorem{remark}[theorem]{Remark}

\makeatletter

\@addtoreset{equation}{section}
\makeatother

\title{Dirichlet Energy of the Scalar Curvature}

\author{Ahmed Mohammed Cherif\footnote{University Mustapha Stambouli Mascara, Faculty of Exact Sciences, Mascara 29000, Algeria. Email: a.mohammedcherif@univ-mascara.dz}}
\date{}
\begin{document}
\maketitle

\begin{abstract}
In this article, we derive the first variation formula for the Dirichlet scalar curvature energy functional on Riemannian manifolds. Motivated by this variational framework, we introduce a natural generalization of Einstein metrics, called Dirichlet--Einstein metrics, and we investigate some of their fundamental properties.
\end{abstract}

\begin{flushleft}
Keywords:  Einstein-Hilbert functional, Dirichlet energy.\\
Subjclass: 53C25, 83C05.
\end{flushleft}

\maketitle

\section{Introduction}

The Einstein-Hilbert functional $\mathcal{E}$ associates to each Riemannian metric $g$ its total scalar curvature, namely
\begin{equation}\label{eq0}
\mathcal{E}:\mathcal{M}\longrightarrow\mathbb{R},\qquad
g\longmapsto \mathcal{E}(g)=\int_M S\,v^g,
\end{equation}
where $\mathcal{M}$ denotes the space of smooth Riemannian metrics on $M$, $S$ is the scalar curvature of $g$, and $v^g$ is the volume form induced by $g$. This functional plays a central role in Riemannian geometry and mathematical physics. In particular, it arises as the action functional of general relativity and has been extensively investigated from both geometric and variational viewpoints \cite{berger,Besse,calabi1,calabi,hilbert,Richard}.

Several generalizations of the Einstein-Hilbert functional have been proposed in the literature. Notable examples include the functionals introduced by Besse \cite{Besse}, Calabi \cite{calabi,calabi1}, and, more recently, Cherif \cite{cherif}.

While the Einstein-Hilbert functional depends only on the scalar curvature itself, it is also natural to study the variation of scalar curvature across the manifold. Recall that a Riemannian metric has constant scalar curvature if and only if the gradient of its scalar curvature vanishes identically. This observation suggests considering the Dirichlet energy of the scalar curvature as a quantitative measure of the deviation from the constant-scalar-curvature condition.

Motivated by this idea, we introduce a new functional on the space of Riemannian metrics defined by
\[
\mathcal{E}_{D}(g)=\int_M |\nabla S|^2\,v^g,
\]
which corresponds to the $L^2$-norm of the gradient of the scalar curvature. This functional is nonnegative and vanishes precisely at metrics of constant scalar curvature. Consequently, it provides a natural variational framework for investigating the distribution of scalar curvature and for introducing a new class of geometrically distinguished metrics. In particular, we derive its first variation formula and we define the associated critical metrics, which we call Dirichlet-Einstein metrics. We then establish several of their fundamental geometric properties.

\section{Dirichlet Energy of Scalar Curvature}

First, we recall some standard definitions and notations. Let $(M,g)$ be an $n$-dimensional Riemannian manifold, and let
$X,X_1,\ldots,X_{q-1},Y,Z\in \Gamma(TM)$.
We denote by $R$, $\operatorname{Ric}$, and $S$ the Riemannian curvature tensor, the Ricci tensor, and the scalar curvature of $(M,g)$, respectively. These tensors are defined by
\begin{equation}\label{curvature}
    R(X,Y)Z=\nabla_X \nabla_Y Z-\nabla_Y \nabla_X Z-\nabla_{[X,Y]}Z,
\end{equation}
\begin{equation}\label{ricci}
    \operatorname{Ric}(X,Y)=g(R(X,e_i)e_i,Y),\quad S=\operatorname{Ric}(e_i,e_i),
\end{equation}
where $\nabla$ denotes the Levi--Civita connection of $g$, and $\{e_1,\ldots,e_n\}$ is a local orthonormal frame on $M$.
Given a smooth function $f$ on $M$, its gradient is $g(\operatorname{grad} f, X)=X(f)$,
its Hessian is  $(\operatorname{Hess} f)(X,Y)=g(\nabla_X \operatorname{grad} f, Y)$, and its Laplacian is defined by
\begin{equation}\label{laplacian}
    \Delta f=-\operatorname{Tr}\left(\operatorname{Hess} f\right).
\end{equation}
The divergence of a $(0,q)$-tensor field $\alpha$ on $M$ is defined by
\begin{equation}\label{divergence}
    (\delta \alpha)(X_1,\ldots,X_{q-1})
    =-(\nabla_{e_i}\alpha)(e_i,X_1,\ldots,X_{q-1}).
\end{equation}
For further details, we refer the reader to \cite{Besse,ON}.


\begin{definition}\label{definition1}
Let $\mathcal{M}$ denote the space of Riemannian metrics on a closed orientable manifold $M$. The Dirichlet scalar curvature energy of a Riemannian manifold $(M,g)$ is the functional
\begin{equation}\label{eq1}
\mathcal{E}_{D}:\mathcal{M}\longrightarrow\mathbb{R}_+,
\qquad
g\longmapsto\mathcal{E}_{D}(g)=\frac{1}{2}\int_M |\nabla S|^2\, v^g,
\end{equation}
where $S$ denotes the scalar curvature of $(M,g)$, $\nabla$ is the gradient with respect to $g$, and $v^g$ is the Riemannian volume form.
\end{definition}

The functional introduced in Definition \ref{definition1} may be regarded as a Dirichlet-type counterpart of the Einstein--Hilbert (or total scalar curvature) functional. While the latter involves the integral of the scalar curvature, the functional $\mathcal{E}_{D}$ measures the $L^2$-norm of its gradient. In particular, metrics of constant scalar curvature are absolute minimizers of $\mathcal{E}_{D}$. The second-order variational properties of the Einstein--Hilbert functional are well understood (see \cite{berger,Besse,calabi1,calabi,hilbert,Koiso1,Koiso,Koiso3,Koiso4,Richard}).

Let $(M,g)$ be a closed orientable Riemannian manifold, and let $(g_t)_{|t|<\varepsilon}$ be a smooth one-parameter family of Riemannian metrics on $M$ such that $g_0=g$. In local coordinates $(x^i)$, the metric $g_t$ can be written as $g_t=g_{ij}(t,x)\,dx^i\otimes dx^j$. We denote by
\[
h=\left.\frac{\partial g_t}{\partial t}\right|_{t=0},
\]
the infinitesimal variation of the metric. Then $h\in \Gamma(S^2T^*M)$ is a smooth symmetric $(0,2)$-tensor field on $M$. The following variational formulas hold.

\begin{theorem}\label{th1}
Let $(M,g)$ be a closed orientable Riemannian manifold and let
$(g_t)_{|t|<\varepsilon}$ be a smooth variation of $g$ with variation tensor $h$. Then, the first variation of the Dirichlet scalar curvature energy in the direction of $h$ is given by
\begin{equation}\label{eq2}
\left.\frac{d}{dt}\right|_{t=0}\mathcal{E}_{D}(g_t)=-\int_M \langle \mathcal{D}(g),h\rangle\,v^g,
\end{equation}
where $\langle \cdot,\cdot\rangle$ denotes the metric induced by $g$ on
$S^2T^*M$, and $\mathcal{D}(g)$ is the gradient of $\mathcal{E}_{D}$ given by
\begin{equation}\label{eq3}
\mathcal{D}(g)=(\Delta S)\operatorname{Ric}-\Big[\Delta(\Delta S)+\frac{1}{4}|\nabla S|^2\Big]g -\operatorname{Hess}(\Delta S)+\frac{1}{2}dS\otimes dS.
\end{equation}
\end{theorem}

\begin{definition}
The symmetric $(0,2)$-tensor $\mathcal{D}(g)$ defined by \eqref{eq3} is called the
Dirichlet--Einstein tensor (or Dirichlet scalar curvature tensor) associated with the Riemannian metric $g$.
\end{definition}

To prove Theorem \ref{th1}, we shall make use of the following classical variation formulas.

\begin{lemma}[\cite{Reto}, \cite{Peter}]\label{lemma1}
Let $(M,g)$ be a Riemannian manifold, and let $(g_t)_{|t|<\varepsilon}$ be a smooth variation of $g$ with variation tensor
$h$.
Then the first variations of the volume element and the scalar curvature are given by
\begin{eqnarray}
\left.\frac{\partial v^{g_t}}{\partial t}\right|_{t=0}
&=&
\frac{1}{2}(\operatorname{Tr}h)\,v^g
=
\frac{1}{2}\langle g,h\rangle\,v^g,\label{eq4}\\
\left.\frac{\partial S_t}{\partial t}\right|_{t=0}
&=&
\Delta(\operatorname{Tr}h)
+\delta(\delta h)
-\langle \operatorname{Ric},h\rangle.\label{eq5}
\end{eqnarray}
\end{lemma}

\begin{proof}[Proof of Theorem \ref{th1}] First, note that
\begin{equation}\label{eq6}
\frac{d}{dt}\mathcal{E}_{D}(g_t)\Big|_{t=0}=\frac{1}{2}\int_M \Big[\frac{\partial |\nabla^{t} S_{t}|^2}{\partial t}v^{g_t}+|\nabla^{t} S_{t}|^2\frac{\partial v^{g_t}}{\partial t}\Big]_{t=0}.
\end{equation}
For any $t\in(-\epsilon,\epsilon)$, we have
\begin{eqnarray}\label{eq7}
  \frac{1}{2}\frac{\partial |\nabla^{t} S_{t}|^2}{\partial t}
   &=&\nonumber \frac{1}{2}\frac{\partial}{\partial t}\Big[g_t^{ij}\frac{\partial S_t}{\partial x_i}\frac{\partial S_t}{\partial x_j}\Big]  \\
   &=&\nonumber \frac{1}{2}\frac{\partial g_t^{ij}}{\partial t}\frac{\partial S_t}{\partial x_i}\frac{\partial S_t}{\partial x_j}
       + g_t^{ij}\frac{\partial}{\partial t}\Big(\frac{\partial S_t}{\partial x_i}\Big)\frac{\partial S_t}{\partial x_j}\\
   &=&\nonumber \frac{1}{2}\frac{\partial g_t^{ij}}{\partial t}\frac{\partial S_t}{\partial x_i}\frac{\partial S_t}{\partial x_j}
       + g_t^{ij}\frac{\partial}{\partial x_i}\Big(\frac{\partial S_t}{\partial t}\Big)\frac{\partial S_t}{\partial x_j}\\
   &=&\nonumber \frac{1}{2}\frac{\partial g_t^{ij}}{\partial t}\frac{\partial S_t}{\partial x_i}\frac{\partial S_t}{\partial x_j}
       + g_t^{ij}\frac{\partial}{\partial x_i}\Big(\frac{\partial S_t}{\partial t}\frac{\partial S_t}{\partial x_j}\Big)
       -g_t^{ij}\frac{\partial S_t}{\partial t}\frac{\partial^2 S_t}{\partial x_i\partial x_j}.\\
\end{eqnarray}
Define $\omega\in\Gamma(T^*M)$ by $\omega=\frac{\partial S_t}{\partial t}\Big|_{t=0}dS$. From (\ref{eq7}), we find that
\begin{eqnarray}\label{eq8}
  \frac{1}{2}\frac{\partial |\nabla^{t} S_{t}|^2}{\partial t}\Big|_{t=0}
   &=& \frac{1}{2}\frac{\partial g_t^{ij}}{\partial t}\Big|_{t=0}\frac{\partial S}{\partial x_i}\frac{\partial S}{\partial x_j}
       -\delta\omega+\frac{\partial S_t}{\partial t}\Big|_{t=0}\Delta S.
\end{eqnarray}
By using $\frac{\partial g_t^{ij}}{\partial t}\Big|_{t=0}=-g^{ia}g^{jb}h_{ab}$ (see \cite{Reto}), equation (\ref{eq8}) becomes
\begin{eqnarray}\label{eq9}
  \frac{1}{2}\frac{\partial |\nabla^{t} S_{t}|^2}{\partial t}\Big|_{t=0}
   &=& -\frac{1}{2}\langle dS\otimes dS,h\rangle
       -\delta\omega+\frac{\partial S_t}{\partial t}\Big|_{t=0}\Delta S.
\end{eqnarray}
By using Lemma \ref{lemma1}, we obtain
\begin{eqnarray}\label{eq10}
\frac{\partial S_t}{\partial t}\Big|_{t=0}\Delta S
   &=& (\Delta S)\Delta(\operatorname{Tr} h)+(\Delta S)\delta(\delta h) -\langle (\Delta S)\operatorname{Ric},h\rangle .
\end{eqnarray}
Calculating in a normal frame at $x \in M$, we have
\begin{eqnarray}\label{eq11}
(\Delta S)\Delta(\operatorname{Tr} h)
   &=&\nonumber -(\Delta S)e_i\big(e_i(\operatorname{Tr} h)\big) \\
   &=&\nonumber -e_i\big((\Delta S)e_i(\operatorname{Tr} h)\big)+e_i(\Delta S)e_i(\operatorname{Tr} h)\\
   &=&\nonumber -e_i\big((\Delta S)e_i(\operatorname{Tr} h)\big)+e_i\big(e_i(\Delta S)\operatorname{Tr} h\big)
   -e_i\big(e_i(\Delta S)\big)\operatorname{Tr} h.\\
\end{eqnarray}
Therfore, the first term in the right-hand side of (\ref{eq10}), is given by
\begin{eqnarray}\label{eq12}
(\Delta S)\Delta(\operatorname{Tr} h)
   &=&\nonumber\delta\big((\Delta S)d(\operatorname{Tr} h)\big)-\delta\big((\operatorname{Tr} h) d(\Delta S)\big)
 +\Delta(\Delta S)\langle g,h\rangle.\\
\end{eqnarray}
Let $f\in C^\infty(M)$ and $\alpha\in \Gamma(T^*M)$. Then the following identity holds (see \cite{Peter,ON})
\begin{equation}\label{eq13}
    \delta(f\alpha)=-\langle df,\alpha\rangle +f\delta \alpha,
\end{equation}
where $\langle df,\alpha\rangle = \alpha(\operatorname{grad} f)$. An application of the above formula yields
\begin{equation}\label{eq14}
    (\Delta S)\delta(\delta h)=\delta\big((\Delta S)\delta h\big)+\langle d(\Delta S),\delta h\rangle.
\end{equation}
By using the following formula (see \cite{Peter})
\begin{equation}\label{eq15}
    (\delta T)(Z)=\delta\big(T(\cdot,Z)\big)+\frac{1}{2}\big\langle T,\mathcal{L}_Z g\big\rangle ,
\end{equation}
where $\mathcal{L}_Z g$ denotes the Lie derivative of $g$ along the vector field $Z\in\Gamma(TM)$ (see \cite{ON}), and $T\in\Gamma(\odot^2T^*M)$. Consequently, we obtain
\begin{eqnarray}\label{eq16}
\langle d(\Delta S),\delta h\rangle
   &=&\nonumber (\delta h)\big( \operatorname{grad} (\Delta S)\big)\\
   &=&\nonumber \delta\big(h(\cdot,\operatorname{grad} (\Delta S))\big)
   +\frac{1}{2}\big\langle h,\mathcal{L}_{\operatorname{grad} (\Delta S)} g\big\rangle \\
   &=& \delta\big(h(\cdot,\operatorname{grad} (\Delta S))\big)
   +\big\langle h, \operatorname{Hess} (\Delta S) \big\rangle ,
\end{eqnarray}
From equations (\ref{eq14}) and (\ref{eq16}), the second term on the left-hand side of (\ref{eq10}) is given by
\begin{eqnarray}\label{eq17}
(\Delta S)\delta(\delta h)
   &=&\nonumber \delta\big((\Delta S)\delta h\big)
     +\delta\big(h(\cdot,\operatorname{grad} (\Delta S))\big)
     +\big\langle h, \operatorname{Hess} (\Delta S) \big\rangle .\\
\end{eqnarray}
By substituting (\ref{eq10}), (\ref{eq12}), and (\ref{eq17}) into (\ref{eq6}) and subsequently applying the divergence theorem (see \cite{baird}) together with Lemma \ref{lemma1}, we obtain Theorem \ref{th1}.
\end{proof}

Moreover, Theorem \ref{th1} admits an extension to non-compact manifolds. From Theorem \ref{th1}, we immediately deduce the following.

\begin{theorem}\label{th2}
A Riemannian metric $g$ on a smooth manifold $M$ is a critical point of the Dirichlet scalar curvature energy if and only if
\begin{equation}\label{eq18}
(\Delta S)\operatorname{Ric}=\Big[\Delta(\Delta S)+\frac{1}{4}|\nabla S|^2\Big]g +\operatorname{Hess}(\Delta S)-\frac{1}{2}dS\otimes dS.
\end{equation}
\end{theorem}

The following result provides a complete characterization of the critical points of the Dirichlet--scalar curvature energy on closed manifolds.

\begin{theorem}\label{th3}
Let $(M,g)$ be a closed Riemannian manifold of dimension $n\neq 6$. Then $g$ is a critical point of the Dirichlet--scalar curvature energy if and only if its scalar curvature is constant.
\end{theorem}

\begin{proof}
By taking the trace of both sides of (\ref{eq18}), we obtain
\begin{eqnarray}\label{eq19}
S(\Delta S)=(n-1)\Delta(\Delta S)+\frac{n-2}{4}\,|\nabla S|^2.
\end{eqnarray}
By the divergence theorem, together with (\ref{eq19}), we deduce that
\begin{eqnarray}\label{eq20}
\int_MS(\Delta S)\,v^g=\frac{n-2}{4}\int_M|\nabla S|^2\,v^g.
\end{eqnarray}
Now, combining (\ref{eq20}) with the identity $\int_M S\,\Delta S\, v^g = \int_M |\nabla S|^2\, v^g,$ we get
\begin{eqnarray}\label{eq20}
\frac{n-6}{4}\int_M|\nabla S|^2\,v^g.
\end{eqnarray}
Since $n \neq 6$, it follows that $S$ is constant on $M$.
\end{proof}

By substituting (\ref{eq19}) into (\ref{eq18}), we obtain the following Proposition.

\begin{proposition}\label{prop1}
A Riemannian metric $g$ on a smooth manifold $M$ of dimension $n>1$ is a critical point of the Dirichlet scalar curvature energy functional if and only if
\begin{equation}\label{eq18-2}
(\Delta S)\operatorname{Ric}=\frac{1}{n-1}\Big[S(\Delta S)+\frac{1}{4}|\nabla S|^2\Big]g +\operatorname{Hess}(\Delta S)-\frac{1}{2}dS\otimes dS.
\end{equation}
\end{proposition}

The following result shows that the Dirichlet--Einstein tensor satisfies a natural conservation law.

\begin{theorem}\label{th4}
Let $(M,g)$ be a Riemannian manifold. Then, the divergence of the Dirichlet--Einstein tensor is zero (that is, $\delta \mathcal{D}(g)=0$).
\end{theorem}

\begin{proof}
In a normal frame $\{e_i\}$ at $x \in M$, with $X = e_j$ fixed, we have
\begin{equation}\label{eq22}
    \delta \mathcal{D}(g) (X)
   = -(\nabla_{e_i}\mathcal{D}(g))(e_i,X)
   = -e_i\big(\mathcal{D}(g)(e_i,X)\big).
\end{equation}
By the definitions of the Dirichlet scalar curvature tensor and the Hessian tensor, we obtain
\begin{eqnarray}\label{eq23}
\mathcal{D}(g)(e_i,X)
&=&\nonumber(\Delta S)\operatorname{Ric}(e_i,X)-\Big[\Delta(\Delta S)+\frac{1}{4}|\nabla S|^2\Big]g(e_i,X) \\
& &-g(\nabla_{e_i}^M\nabla(\Delta S),X)+\frac{1}{2}e_i(S)g(X,\nabla S).
\end{eqnarray}
Substituting (\ref{eq23}) into (\ref{eq22}),  we conclude that
\begin{eqnarray}\label{eq24}
\delta \mathcal{D}(g) (X)
   &=&\nonumber-e_i(\Delta S)\operatorname{Ric}(e_i,X)-(\Delta S)e_i(\operatorname{Ric}(e_i,X))\\
   & &\nonumber   +\Big[e_i(\Delta(\Delta S))+\frac{1}{4}e_i(|\nabla S|^2)\Big]g(e_i,X)+g(\nabla_{e_i}^M\nabla_{e_i}^M\nabla(\Delta S),X) \\
   & &-\frac{1}{2}e_i(e_i(S))g(X,\nabla S)-\frac{1}{2}e_i(S)g(X,\nabla_{e_i}^M\nabla S).
\end{eqnarray}
Using the definition of the gradient operator and the definition of the divergence, equation (\ref{eq24}) becomes
\begin{eqnarray}\label{eq25}
\delta \mathcal{D}(g) (X)
   &=&\nonumber-\operatorname{Ric}(\nabla(\Delta S),X)+(\Delta S)(\delta\operatorname{Ric})(X)\\
   & &\nonumber   +X(\Delta(\Delta S))+\frac{1}{4}X(|\nabla S|^2)+g(\operatorname{Tr}\,(\nabla^M)^2\nabla(\Delta S),X) \\
   & &+\frac{1}{2}\Delta(S)g(X,\nabla S)-\frac{1}{2}g(X,\nabla_{\nabla S}^M\nabla S).
\end{eqnarray}
Theorem \ref{th4} follows from \eqref{eq25} and the following formulas
\begin{eqnarray*}
(\delta \operatorname{Ric})(X)&=&-\frac{1}{2}X(S), \quad\nabla_{\nabla S}^M\nabla S = \frac{1}{2}\nabla|\nabla S|^2, \\
g(\operatorname{Tr}\,(\nabla^M)^2\nabla(\Delta S),X) &=& \operatorname{Ric}(\nabla(\Delta S),X)-X(\Delta(\Delta S)).
\end{eqnarray*}
\end{proof}

\begin{remark}
Let $M$ be an orientable manifold. We define
\[
\mathcal{M}_{c}=\{g\in \mathcal{M}\mid \operatorname{Vol}(M,g)=\int_M v^g=c\},
\]
for some constant $c>0$. This is a codimension $1$ submanifold of $\mathcal{M}$, and its tangent space at $g\in\mathcal{M}_{c}$ is given by
\[
T_g\mathcal{M}_{c}=\left\{T\in \Gamma(\odot^2 T^*M)\ \middle|\ \int_M \langle g,T\rangle\, v^g=0\right\}.
\]
A Riemannian metric $g$ is a critical point of $\mathcal{E}_{D}\big|_{\mathcal{M}_{c}}$ if and only if $\mathcal{D}(g)$ is orthogonal to $T_g\mathcal{M}_{c}$, that is, $\mathcal{D}(g)=\lambda\,g$ for some constant $\lambda$.
\end{remark}

\begin{definition}
A Riemannian metric $g$ satisfying $\mathcal{D}(g)=\lambda\,g$ is called a Dirichlet--Einstein metric.
In the special case $\lambda=0$, the metric $g$ is called a flat Dirichlet--Einstein metric.
\end{definition}

\begin{remark}
A flat Dirichlet--Einstein metric, namely a metric $g$ satisfying
$\mathcal{D}(g)=0,$ is precisely a critical point of the Dirichlet--scalar curvature energy. In particular, every Riemannian metric of constant scalar curvature is a flat Dirichlet--Einstein metric, since the vanishing of the gradient of the scalar curvature implies
$\mathcal{D}(g)=0$. Thus, the class of flat Dirichlet--Einstein metrics contains all constant-scalar-curvature metrics.
\end{remark}

\begin{example}[Existence of flat Dirichlet--Einstein metrics on warped products]
Let $(\mathbb{E}^{n-1},h)$ be an $(n-1)$-dimensional flat Einstein Riemannian manifold.
Let $M = \mathbb{R}_{+}^{*} \times \mathbb{E}^{n-1}$ be equipped with the warped product metric
$g = dx^2 + f(x)^2\, h,$ where the warping function is given by $f(x) = x^\alpha$ for some constant $\alpha \in \mathbb{R}^*$.
Let $X,Y\in\Gamma(T\mathbb{E}^{n-1})$. A direct computation shows that
\begin{eqnarray*}
\operatorname{Ric}(\partial_x,\partial_x)&=& -(n-1)\frac{f''}{f},\quad \operatorname{Ric}(\partial_x,X)=0, \\
\operatorname{Ric}(X,Y)&=& -\left[ff''+(n-2)(f')^2\right]h(X,Y),\\
S&=&  -\frac{(n-1)}{f^2}\left[2ff''+(n-2)(f')^2\right],\\
\Delta S &=& -S'' -(n-1)\frac{f'}{f}S',\\
\Delta(\Delta S) &=& -(\Delta S)'' -(n-1)\frac{f'}{f}(\Delta S)' ,\\
|\nabla S|^2&=&(S')^2,\\
dS \otimes dS&=&(S')^2 \, dx^2,\\
\operatorname{Hess}(\Delta S)&=& (\Delta S)''\, dx^2+f f' (\Delta S)'\, h.
\end{eqnarray*}
Substituting these formulas into \eqref{eq3}, together with $f(x)=x^\alpha$, we obtain
\begin{eqnarray*}
  \mathcal{D}(g)(\partial_x,\partial_x)
    &=& ( n-1 ) ^{2}{\alpha}^{2} ( n\alpha-2 ) \big[ 2( n-1 ) {\alpha}^{2}+ ( 7n-12 ) \alpha-20\big] {x}^{-6},\\
  \mathcal{D}(g)(\partial_x,X)
    &=&0,\\
  \mathcal{D}(g)(X,Y)
    &=& \alpha( n-1 )(n\alpha-2)\big((n-1)\alpha-6\big)\big[ 2( n-1 ) {\alpha}^{2}\\
    & &+ ( 7n-12 ) \alpha-20\big] {x}^{2\alpha-6}h(X,Y).
\end{eqnarray*}
Therefore, $\mathcal{D}(g)=0$ if and only if $n\alpha-2=0$ or $2( n-1 ) {\alpha}^{2}+ ( 7n-12 ) \alpha-20=0$, that is
$\alpha=2/n$ (in which case the scalar curvature vanishes) or $\alpha=\big(-7n+12\pm\sqrt{49n^2-8n-16}\big)/\big(4n-4\big)$ (in which case the scalar curvature is non-constant).
\end{example}

\begin{remark}
Every Einstein metric is a flat Dirichlet--Einstein metric.
The preceding example shows that there exist Riemannian flat Dirichlet--Einstein metrics that are not Einstein metrics.
\end{remark}


\subsection*{Conflict of interest statement}
The author declares no conflict of interest.

\subsection*{Data availability}
Not applicable.


\begin{thebibliography}{99}

\bibitem{baird} P. Baird, J.C. Wood, {\it Harmonic Morphisms between Riemannain Manifolds}, Clarendon Press, Oxford, 2003.

\bibitem{berger} M. Berger,  {\it Quelques formules de variation pour une structure riemannienne}, Ann. Sci. Ecole Norm. Sup., {\bf4} (1970), no. 3, 285-294.

\bibitem{Besse} A. L. Besse,  {\it Einstein Manifolds}, Springer-Verlag, Berlin, 1987.

\bibitem{calabi1} E. Calabi,  {\it Extremal K\"{a}hler metrics}, Ann. Math. Stud., {\bf102} (1982), 259-290.

\bibitem{calabi} E. Calabi,  {\it Extremal K\"{a}hler metrics}, II, in: Differential Geometry and Complex Analysis, Springer, Berlin (1985), 95-114.


\bibitem{hilbert} D. Hilbert,  {\it Die grundlagen der Physik}, Nachr. Ges. Wiss. Göttingen (1915), 395-407.


\bibitem{Koiso1} N. Koiso, {\it Nondeformability of Einstein metrics}, Osaka J. Math., {\bf 15} (1978), 419--433.

\bibitem{Koiso} N. Koiso, {\it On the second derivative of the total scalar curvature}, Osaka J. Math., {\bf16} (1979), 413-421.

\bibitem{Koiso3} N. Koiso, {\it Rigidity and stability of Einstein metrics-the case of compact symmetric spaces}, Osaka
J. Math., {\bf17} (1980), 51--73.

\bibitem{Koiso4} N. Koiso, {\it Rigidity and infinitesimal deformability of Einstein metrics}, Osaka J. Math., {\bf19}
(1982), 643--668.


\bibitem{cherif} A. Mohammed Cherif, {\it Second Variation of F-Einstein-Hilbert Functional}, Calculation J., accepted for publication 2026.


\bibitem{Reto} R. Müller,  {\it Differential Harnack Inequalities and the Ricci Flow}, European Mathematical Society, 2006.

\bibitem{ON} O'Neil,  {\it Semi-Riemannian Geometry}, Academic Press, New York, 1983.

\bibitem{Richard} R. M. Schoen,  {\it Variational theory for the total scalar curvature functional for Riemannian metrics and related topics}, Topics in calculus of variations, Lect. 2nd Sess., Montecatini/Italy (1987), Lect. Notes Math. {\bf1365}(1989), 120-154.


\bibitem{Peter} P. Topping,  {\it Lectures on the Ricci Flow}. Number 325 in London Mathematical Society Lecture Note Series. Cambridge University Press, October, 2006.

\end{thebibliography}
\end{document}